\documentclass[a4paper,12pt,intlimits,oneside]{amsart}

\usepackage{amssymb,latexsym,amsmath,mathrsfs}
\usepackage{amsfonts,amsbsy,bm}
\usepackage{color}
\usepackage[margin=1.5cm]{geometry}

\usepackage{amsmath}
\usepackage{amssymb}
\usepackage{amsfonts}
\usepackage{amsthm,amscd}

\newtheorem{theorem}{Theorem}[section]
\newtheorem{proposition}[theorem]{Proposition}

\newtheorem{lemma}[theorem]{Lemma}

\newtheorem{remark}[theorem]{Remark}

\theoremstyle{definition}
\newcommand{\comment}[1]{}

\numberwithin{equation}{section}

\newcommand{\epf}{ $\Box$\medskip}

\theoremstyle{definition}

\begin{document}
\title{Factorization of Hardy-Orlicz Space on the  Disk and applications to Hankel Operators}
\author{Jean-Marcel Tanoh Dje and Justin Feuto}
\address{Unit\'e de Recherche et d'Expertise Num\'erique, Universit\'e Virtuelle de C\^ote d'Ivoire, Cocody II-Plateaux - 28 BP 536 ABIDJAN 28}
\email{{\tt tanoh.dje@uvci.edu.ci}}
\address{Laboratoire de Math\'ematiques et Applications, UFR Math\'ematiques$-$Informatique, Universit\'e F\'elix Houphou\"et-Boigny Abidjan-Cocody, 22 B.P 1194 Abidjan 22. C\^ote d'Ivoire}
\email{{\tt justfeuto@yahoo.fr}}


\date{}

\begin{abstract}
In this work, we prove that the product of a function belonging to a Hardy-Orlicz space  $H^{\Phi_{1}}$
 and a function from another Hardy-Orlicz space 
 $H^{\Phi_{2}}$  belongs to a third Hardy-Orlicz space $H^{\Phi_{3}}$. Moreover, we establish the converse: any holomorphic function in the space $H^{\Phi_{3}}$ can be expressed as the product of two functions, one from 
 $H^{\Phi_{1}}$  and the other from $H^{\Phi_{2}}$. Subsequently, we use this factorization result in Hardy-Orlicz spaces to study the continuity of the Hankel operator in these spaces. More specifically, we provide gain and loss estimates for the norms of the Hankel operator in the context of analyzing its continuity in Hardy-Orlicz spaces.
\end{abstract}

\maketitle


\section{Introduction and statement of main results.}

Recall that the classical Hardy space $H^p$, $0 < p < \infty$, on the unit disc $\mathbb D := \left\lbrace z \in\mathbb C : \vert z\vert < 1\right\rbrace $
is defined as the space of holomorphic functions $f$ satisfying $$\Vert f\Vert_{H^p}:=\sup_{0\leq r<1}\left( \frac{1}{2\pi}\int^{2\pi}_0\vert f(re^{it})\vert^pdt\right) ^{\frac{1}{p}}<\infty.$$

For a sequence  $(z_{n})_{n\geq 1}$ of elements of 
$\mathbb{D}$  such that 
$    \sum_{n\geq 1}(1-|z_{n}|)< \infty$,  
the Blaschke product on $\mathbb{D}$  associated to $(z_{n})_{n\geq 1}$ is the function $B$ defined by
\begin{equation}\label{eq:suqiaqqaqaq8n}
B(z)= \prod_{n=1}^{+\infty} \frac{|z_{n}|}{z_{n}}\frac{z_{n}-z}{1-\overline{z_{n}}z}, ~~\forall~z \in \mathbb{D}.
\end{equation}
It is a classical result that any function in the Hardy space $H^{p}$ on the unit disk   can be factored as $f = Bg$ with $\|f\|_{H^{p}}=\|g\|_{H^{p}}$, where $B$ is a Blaschke
product and $g$ is an $H^{p}-$ function with no zero on the unit disk. An immediate consequence of the above result is that for $0<p<\infty$, any function $f$ in the Hardy space $H^{p}$ can be writen as 
$f = f_{1}f_{2}$ with $f_{1} \in  H^{p_{1}}$, $f_{2} \in  H^{p_{2}}$ and $\|f_{1}\|_{H^{p_{1}}}.\|f_{2}\|_{H^{p_{2}}}=\|f\|_{H^{p}}$, where $p_{1}$ and $p_{2}$ are two positive reals numbers satisfying the condition $1/p = 1/p_{1} + 1/p_{2}$ (see \cite{Pduren2, Jbgarnett, javadmas, wrudin}).

\medskip

Recall that a function $\Phi : [0,\infty) \rightarrow [0,\infty)$ is called an growth function if it is nondecreasing, $\lim_{t \to 0} \Phi(t) =\Phi(0)= 0$, $\Phi(t) > 0$
for $t \in (0,\infty)$ and $\lim_{t \to \infty} \Phi(t) = \infty$.  The growth function  $\Phi$ is said to be of upper type (resp. lower type) if there exists  $p \in (0,\infty)$ and a constant $C>1$ such that for all $t \in  [1,\infty)$ (resp. $t \in  [0,1]$) and $s \in  [0,\infty)$,
\begin{equation}\label{eq:sui8n}
\Phi(st)\leq Ct^{p}\Phi(s).\end{equation}

Let  $\Phi_{1}$ and  $\Phi_{2}$ be two positive functions on  $[0,\infty)$. We say that  $\Phi_{1}$ and  $\Phi_{2}$ are equivalent and we denote  $\Phi_{1} \sim \Phi_{2}$ if there exists a constant $c > 0$ such that
\begin{equation}\label{eq:equivalent}
c^{-1}\Phi_{1}(c^{-1}t) \leq \Phi_{2}(t)\leq c\Phi_{1}(ct), ~~ \forall~ t > 0.\end{equation} 

Let  $\Phi$ be a lower type growth function. Recall that the Orlicz space on the complex unit circle $\mathbb{T}:=\{ z\in \mathbb{C} : |z |=1     \}$  is the space  $L^{\Phi}(\mathbb{T})$ of  measurable functions $g : \mathbb{T} \longrightarrow \mathbb{C}$ which satisfy
$$ \|g\|_{L^{\Phi}}^{lux}:=\inf\left\{\lambda>0 :  \int_{0}^{2\pi}\Phi\left(\frac{|g(e^{i\theta})|}{\lambda}\right)\frac{d\theta}{2\pi} \leq 1  \right\}< \infty. $$
Volberg and Tolokonnikov obtained also a strong factorization of Orlicz spaces in \cite{voltiko}. Their result can be reformulated as follows: the Orlicz space $L^{\Phi_{3}}$ is equal to the product of the Orlicz spaces $L^{\Phi_{1}}$ and $L^{\Phi_{2}}$ if and only if the function $\Phi_{3}^{-1}$ is equivalent to the product of the functions $\Phi_{1}^{-1}$ and $\Phi_{2}^{-1}$, where the $\Phi_{j}$ are lower type growth functions and 
$\Phi_{j}^{-1}$ represents the inverse function of 
$\Phi_{j}$, for $j\in \{1,2,3\}$. 

\medskip
The objective of this work is to propose a generalization of the Riesz factorization, originally obtained in classical Hardy spaces, to Hardy-Orlicz spaces. Furthermore, we aim to establish an analogue of the strong factorization, developed by Volberg and Tolokonnikov in Orlicz spaces, for Hardy-Orlicz spaces.


\medskip

Let  $\Phi$ be a lower type growth function. The Hardy-Orlicz space on $\mathbb{D}$,    $H^{\Phi}(\mathbb{D})$ is the space of holomorphic functions $G$ on  $\mathbb{D}$ 
which satisfy
$$ \|G\|_{H^{\Phi}}^{lux}:=\sup_{0\leq r<1 }\|G_{r}\|_{L^{\Phi}}^{lux}< \infty,   $$
where  $G_{r}$  is the function defined by
\begin{equation}\label{eq:equivaaqlent}
G_{r}(e^{i\theta}):=G(re^{i\theta}), ~~ \forall~ \theta \in \mathbb{R}.\end{equation}

Our first main result can be formulated as follow:
  
\begin{theorem}\label{pro:main2a}
Let  $\Phi$ be a growth function of lower type. Let  $0\not\equiv G\in H^{\Phi}(\mathbb{D})$ and   $B$  the Blaschke product associated with the zeros sequence of $G$. The function $G/B$ belongs to   $H^{\Phi}(\mathbb{D })$ and
$\|G\|_{H^{\Phi}}^{lux}\approx\left\|G/B\right\|_{H^{\Phi}}^{lux}$.
\end{theorem}

We say that an analytic function  $G$ on $\mathbb{D}$ is 
\begin{itemize}
\item[(i)] an inner function  if $G\in H^{\infty}(\mathbb{D})$ 
and 
$ \left|\lim_{r\to 1}G(re^{it})\right|=1,$
for almost all $t\in \mathbb{R}$,
\item[(ii)] an outer function if 
\begin{equation}\label{eq:phiaq6qal3de}
 G(z)= \exp\left\{ \frac{1}{2\pi}\int_{-\pi}^{\pi}\frac{e^{it}+z}{e^{it}-z}\log|g(e^{it})|dt  \right\}, ~~\forall~z \in \mathbb{D},\end{equation}
where  $g$ is  a measurable function on $\mathbb{T}$ such that  $\log |g| \in L^{1}(\mathbb{T})$. 
 \end{itemize}
Generally we say that $G$ is the outer function associated with $|g|$.

 \begin{theorem}\label{pro:main0apaaal0}
 Let  $\Phi$ be a growth function of lower type. For  $0\not\equiv G\in H^{\Phi}(\mathbb{D})$, there exists a unique decomposition of the form $G=I_{G}O_{G}$,
 where $I_{G}$ is an inner function and $O_{G}$ is an outer function belonging to $H^{\Phi}(\mathbb{D})$. Moreover, 
 $ \|G\|_{H^{\Phi}}^{lux} \approx\left\|O_{G}\right\|_{H^{\Phi}}^{lux}.$
 \end{theorem}
 
 When  $\Phi(t)=t^{p}$ with  $0< p  < \infty$, Theorem \ref{pro:main2a} coincides with a classical result by Riesz for Hardy spaces. This result, which can be found on Wikipedia or in works accessible to the general public, such as Rudin’s book \cite{wrudin}, represents the first factorization theorem for Hardy spaces. Thanks to this theorem, it becomes possible to work in Hardy spaces by focusing solely on analytic functions that do not vanish on  $\mathbb{D}$. This assumption allows for a strong factorization in the framework of classical Hardy spaces, namely  $H^{p}=H^{p_{1}}.H^{p_{2}}$, where  $1/p = 1/p_{1} + 1/p_{2}$. Indeed, given an analytic function $f$ belonging to $H^{p}$ and not vanishing on $\mathbb{D}$, the function  $f^{p/p_{j}}$ is also analytic on  $\mathbb{D}$ and belongs to 
$H^{p_{j}}$, for $j\in \{1,2\}$, which implies that $f =f^{p/p_{1}}f^{p/p_{2}}$.

However, in the context of Hardy-Orlicz spaces, this condition is not sufficient to guarantee such a factorization. Indeed, it seems difficult to construct an analytic function on $\mathbb{D}$ from an analytic function $f$ on $\mathbb{D}$ and a growth function $\Phi$ when the latter is not equivalent to a power function. Thus, the natural method used by Riesz to obtain strong factorization does not seem applicable in this general case.

It is worth noting that the result stated in Theorem  \ref{pro:main0apaaal0} generalizes the canonical factorization of classical Hardy spaces, as can be found in the references (\cite{Pduren2, Jbgarnett, javadmas, wrudin}). This result also enables strong factorization in the classical setting by factoring the outer function associated with this decomposition. To extend this strong factorization to Hardy-Orlicz spaces, this approach seems to work similarly. In the following result, we reformulate this generalization of the strong factorization obtained by Riesz in the classical case as follows:

 \begin{theorem}\label{pro:main2aqop}
Let $\Phi_{1}$,  $\Phi_{2}$ and $\Phi_{3}$ be growth functions of the lower type such that  $\Phi_{3}^{-1} \sim \Phi_{1}^{-1}.\Phi_{2}^{-1}$, where  $\Phi_{j}^{-1}$ is the inverse function of  $\Phi_{j}$, for  $j\in \{1,2,3\}$. For all $G_{1}\in   H^{\Phi_{1}}(\mathbb{D})$ and  $G_{2}\in H^{\Phi_{2}}(\mathbb{D})$, the product   $G_{1}G_{2} \in H^{\Phi_{3}}(\mathbb{D})$.  Conversely, for $ G\in H^{\Phi_{3}}(\mathbb{D})$, there exist $G_{1}\in H^{\Phi_{1}}(\mathbb{D})$ and $G_{2}\in H^{\Phi_{2}}(\mathbb{D})$ such that
 $G=G_{1}G_{2}$.
 Moreover, $ \|G\|_{H^{\Phi_{3}}}^{lux} \approx \|G_{1}\|_{H^{\Phi_{1}}}^{lux}.\|G_{2}\|_{H^{\Phi_{2}}}^{lux}.$
 \end{theorem}
 
A natural application of such factorizations is the characterization of symbols of bounded Hankel operators. 
 
\medskip
 
The orthogonal projection of $L^{2}(\mathbb{T})$ onto $H^{2}(\mathbb{D})$ is called the  Szeg\"o projection and denoted $\mathcal{P}$. It is given by $$  \mathcal{P}(g)(z)= \frac{1}{2\pi}\int_{-\pi}^{\pi}  \frac{g(e^{i\theta})}{1- z e^{-i\theta}}d\theta, ~~\forall~z \in \mathbb{D} .  $$ 
 For $b \in H^{2}(\mathbb{D})$, the  Hankel operator with symbol $b$ is defined  by 
 $$ h_{b}(g)= \mathcal{P}(b \overline{g}),    $$
with $g$ a bounded holomorphic function $\mathbb{D}$. 

Using simple techniques, Bonami et al proved in \cite{BoGreseh} that the Hankel operator,  $h_{b}$ is bounded from $H^{1}$  to $H^{1}$
 if and only if $b$ belongs to the space $BMOA(\rho)$, where $\rho(t)=\frac{1}{\log(4/t)}$  (we will define the space  $BMOA(\rho)$ more precisely later in this section). Next, in \cite{BoGre}, by applying weak factorization results, Bonami and Grellier showed that the Hankel operator  $h_{b}$ is bounded from  $H^{\Phi}$  to 
 $H^{1}$, where $\Phi$ is a concave growth function. This work was extended by Bonami and Sehba in \cite{BoSehb}, where they demonstrated that  $h_{b}$ is bounded from 
  $H^{\Phi_{1}}$ to  $H^{\Phi_{2}}$, where $\Phi_{1}$ and $\Phi_{2}$ are concave growth functions. More recently, in \cite{sehbaedgc1}, Sehba and Tchoundja proved that  $h_{b}$  is bounded from  $H^{\Phi_{1}}$  to  $H^{\Phi_{2}}$, in the following cases: either $\Phi_{1}$ is concave and $\Phi_{2}$ is convex, or both 
$\Phi_{1}$ and $\Phi_{2}$ are convex growth functions. However, their results are limited to the case where  $H^{\Phi_{2}}$  is a subspace of  $H^{\Phi_{1}}$. The reverse case, where  $H^{\Phi_{1}}$  is a subspace of 
 $H^{\Phi_{2}}$, with both $\Phi_{1}$ and $\Phi_{2}$
  being convex growth functions, was not addressed in their work and appears to be absent from the literature.

This is precisely the case we study in this work. We propose an extension of the results obtained by these authors and also present estimates for the norm of the Hankel operator in each of the cases discussed.

 \medskip
 
Let  $\Phi_{1}$ and $\Phi_{2}$ be two growth functions of lower type and $b$ be a holomorphic function on  $\mathbb{D}$. We say that the  Hankel operator $h_{b}$ with symbol $b$ is bounded from  $H^{\Phi_{1}}(\mathbb{D})$ to $H^{\Phi_{2}}(\mathbb{D})$ if there exists a constant $C>0$  such that
\begin{equation}\label{eq:ineitedaaqpqqehay}
\|h_{b}(g)\|_{H^{\Phi_{2}}}^{lux} \leq C \|g\|_{H^{\Phi_{1}}}^{lux}.
\end{equation}
The norm of $h_{b}$ is given by
$$ \|h_{b}\| =\|h_{b}\|_{H^{\Phi_{1}}\to H^{\Phi_{2}}} := \sup\{\|h_{b}(g)\|_{H^{\Phi_{2}}}^{lux} : \|g\|_{H^{\Phi_{1}}}^{lux} \leq 1   \}. $$
 
We say that the Hankel operator  $h_{b}$ is bounded to loss (resp. gain) from  $H^{\Phi_{1}}(\mathbb{D})$ to $H^{\Phi_{2}}(\mathbb{D})$, if $H^{\Phi_{1}}(\mathbb{D})\subset H^{\Phi_{2}}(\mathbb{D})$ (resp. $H^{\Phi_{2}}(\mathbb{D})\subset H^{\Phi_{1}}(\mathbb{D})$).
  
The first result we obtain regarding the continuity of the Hankel operator in Hardy-Orlicz spaces is as follows:

\begin{theorem}\label{pro:mainfqmaqppaaqqmaqq5}
Let $\Phi_{j}$ be a growth function of lower type $p_{j}$,  for $j \in \{1,2\}$. Suppose that  $\Phi_{2}$ is also of upper type $q_{2}$ such that   $1<p_{2} \leq q_{2} < p_{1}< \infty$.  Then for  $b$  a holomorphic function on  $\mathbb{D}$, the Hankel operator $h_{b}$ is bounded from  $H^{\Phi_{1}}(\mathbb{D})$ to $H^{\Phi_{2}}(\mathbb{D})$ if and only if  $b \in H^{\Phi_{3}}(\mathbb{D})$, where $\Phi_{3}$ is a lower type growth function such that $ \Phi_{3}^{-1}(t) \sim \frac{\Phi_{2}^{-1}(t)}{\Phi_{1}^{-1}(t)}$.  Moreover,
\begin{equation}\label{eq:ineitmqqqaqqehay}
\|h_{b}\|\approx\|b\|_{H^{\Phi_{3}}}^{lux}.
\end{equation}
\end{theorem}

In Theorem \ref{pro:mainfqmaqppaaqqmaqq5}, the result we obtain consists of a loss estimate of the norm of the Hankel operator when studying its continuity of the Hardy-Orlicz space 
$H^{\Phi_{1}}$ to the Hardy-Orlicz space $H^{\Phi_{2}}$. This result has both advantages and limitations. The advantage is that the starting function $\Phi_{1}$ is of lower type only, which allows our result to include functions of exponential type (e.g. $t\mapsto \exp(t )-1$, which are never not of upper type). Thus, we can study the continuity of the Hankel operator from $H^{\Phi_{1}}$ to $H^{\Phi_{2}}$, even when the function $\Phi_{1}$ is of exponential type. However, a limitation of this result is that it does not cover the case where the functions $\Phi_{1}$ and 
$\Phi_{2}$ are identical or even equivalent.

\medskip
 
In the following result, we present another result on the continuity of the Hankel operator, which accounts for the case where the functions $\Phi_{1}$ and $\Phi_{2}$ are identical or equivalent.

Let  $\varrho$ be a positive function on $[0,\infty)$ such that $\varrho (t)>0$, for  $t>0$. We define the space $BMO(\varrho)$ as the space of    $g \in L^{2}(\mathbb{T})$  which satisfy
 $$  \|g\|_{BMO(\varrho)}:= \sup_{I \subset \mathbb{T}} \frac{1}{\varrho(|I|)}\left(\frac{1}{|I|}\int_{I}|g(z)-m_{I}(g)|^{2}dz \right)^{1/2} < \infty,   $$
 where the supremum is taking on all intervals $I \subset \mathbb{T}$,  and $m_{I}(g):=\frac{1}{|I|}\int_{I}g(s)ds$. Here, for any measurable set $E\subset \mathbb{T}$, $|E|$ denotes the Lebesgue measure of $E$. The space $BMOA(\varrho)$ is the space of holomorphic functions   $G \in H^{2}(\mathbb{D})$  which satisfy
 $$ \|G\|_{BMOA(\varrho)}:=\sup_{0\leq r<1 }\|G_{r}\|_{BMO(\varrho)}< \infty,   $$
 where  $G_{r}$  is the function defined in (\ref{eq:equivaaqlent}). We will simply denote $BMO(\varrho)$ and $BMOA(\varrho)$ by $BMO(\mathbb{T})$ and $BMOA(\mathbb{D})$, respectively, when $\varrho \equiv  1$.

\begin{theorem}\label{pro:madgaaqqmaqq5}
Let $\Phi_{j}$ be a growth function of both lower type $p_{j}$ and upper type $q_{j}$ and let $\varrho_{j}(t):=\frac{1}{t\Phi_{j}^{-1}(1/t)}$, for $j \in \{1,2\}$. Suppose that  $0<p_{1} \leq q_{1} \leq p_{2}$ and  $1<p_{2} \leq q_{2} < \infty$. Then for  $b$  a holomorphic function on  $\mathbb{D}$, the Hankel operator $h_{b}$ is bounded from  $H^{\Phi_{1}}(\mathbb{D})$ to $H^{\Phi_{2}}(\mathbb{D})$ if and only if  $b \in BMOA(\varrho)$, where  
$\varrho:= \frac{\varrho_{1}}{\varrho_{2}}.$ Moreover,
\begin{equation}\label{eq:ineitmqaaaqqaaqqqehqay}
\|h_{b}\|\approx\|b\|_{BMOA(\varrho)}.
\end{equation}
\end{theorem}

In Theorem \ref{pro:madgaaqqmaqq5}, we obtain a gain estimate for the norm of the Hankel operator in the context of its continuity between the Hardy-Orlicz spaces 
 $H^{\Phi_{1}}$ and  $H^{\Phi_{2}}$. The advantage of the result obtained in this theorem is that it does not require imposing convexity or concavity on the function 
$\Phi_{1}$. It is sufficient for $\Phi_{1}$ to be both of lower type $p_{1}$ and upper type $q_{1}$. However, it is not necessary for $p_{1}$ and $q_{1}$ to both be less than $1$ or greater than $1$.

\medskip

We use the  abbreviation $\mathrm{ A}\lesssim \mathrm{ B}$ for inequalities $\mathrm{ A}\leq C\mathrm{ B}$, where $C$ is a positive constant independent of the main parameters. If $\mathrm{ A}\lesssim \mathrm{ B}$ and $\mathrm{ B}\lesssim \mathrm{ A}$, then we write $\mathrm{ A} \approx \mathrm{ B}$.
In all what follows, the letter $C$ will be used for non-negative constants independent of the relevant variables that may change from one occurrence to another. Constants with subscript, such as $C_{s}$, may also change in different occurrences, but depend on the parameters mentioned in it.

\section{Some definitions and useful properties.}

\subsection{Growth functions.} 

Let  $\Phi$ be a growth function. We say that  $\Phi$ satisfies the $\Delta_{2}-$condition (or $\Phi \in \Delta_{2}$) if there exists a constant $K > 1$ such that
\begin{equation}\label{eq:delta2}
\Phi(2t) \leq K \Phi(t),~ \forall~ t >  0.\end{equation}
We say also that  $\Phi$ satisfies the $\nabla_{2}-$condition (or $\Phi \in \nabla_{2}$) if there exists $C > 1$ such that 
\begin{equation}\label{eq:delmta2}
\Phi(t) \leq \frac{1}{2C} \Phi(Ct),~ \forall~ t >  0.\end{equation}

We can find most of the following results in (\cite{raoren, rao68ren}):

Let  $\Phi$ be a convex growth function such that  $\lim_{t \to 0}\frac{\Phi(t)}{t}=0$ and $\lim_{t \to \infty}\frac{\Phi(t)}{t}=+\infty$. There exists  $\varphi$ a positive, left continuous and non-decreasing function on $[0, \infty)$ such that $\varphi(0)=0$ and $\lim_{t \to \infty}\varphi(t)=\infty$ and, 
$$   \Phi(t)=\int_{0}^{t}\varphi(s)ds, ~~\forall~t\geq 0.
  $$
The complementary function of $\Phi$ is the function $\Psi$ defined by
$$ \Psi(s)=\sup_{t\geq 0}\{st-\Phi(t) \}, ~ \forall~  s \geq 0.       $$
Note that function  $\Psi$ has the same properties as function $\Phi$ and, the complementary of $\Psi$ is $\Phi$. Put
$$  a_\Phi:=\liminf_{t\to \infty}\frac{t\varphi(t)}{\Phi(t)}
   \hspace*{1cm}\textrm{and} \hspace*{1cm} b_\Phi:=\limsup_{t\to \infty}\frac{t\varphi(t)}{\Phi(t)}, $$
 and similarly $a_\Psi$, $b_\Psi$ be defined. We have the following properties:
\begin{itemize}
\item[(i)] For all $t>0$,  $t<\Phi^{-1}(t)\Psi^{-1}(t) \leq 2t$.
\item[(ii)]  $b_\Phi< \infty$ if and only if $a_\Psi>1$. Moreover, $\frac{1}{a_\Psi}  + \frac{1}{b_\Phi} =1.$
\item[(iii)] $\Phi \in \Delta_{2} \cap \nabla_{2}$  if and only if 
 $1< a_\Phi \leq b_\Phi< \infty$.
\end{itemize}

\begin{proposition}\label{pro:mainfaaqaaqldq5}
Let  $\Phi$ be a convex growth function such that $\lim_{t \to 0}\frac{\Phi(t)}{t}=0$ and $\lim_{t \to \infty}\frac{\Phi(t)}{t}=+\infty$. If $0 < a_{\Phi} \leq b_{\Phi}< \infty$ then  $\Phi$ is respectively lower type  $a_{\Phi}$ and upper type  $b_{\Phi}$.
\end{proposition}

\begin{proof}
It will suffice to prove that the function $t\mapsto \frac{\Phi(t)}{t^{a_\Phi}}$ (resp.  $t\mapsto \frac{\Phi(t)}{t^{b_\Phi}}$) is non-decreasing (resp. non-increasing) on$(0, \infty)$.

\medskip

For $0<t_{1} \leq t_{2}$, we have
$$ \log\left(  \frac{t_{2}}{t_{1}} \right)^{a_{\Phi}} = \int_{t_{1}}^{t_{2}}a_{\Phi}\frac{dt}{t} \leq \int_{t_{1}}^{t_{2}}\frac{\varphi(t)}{\Phi(t)}dt = \log\left(  \frac{\Phi(t_{2})}{\Phi(t_{1})} \right).  $$
We deduce that the function $t\mapsto \frac{\Phi(t)}{t^{a_\Phi}}$ is non-decreasing on $(0, \infty)$.

\medskip

In the same way, we show that the function $t\mapsto \frac{\Phi(t)}{t^{b_\Phi}}$ is non-increasing on $(0, \infty)$.
\end{proof}

\begin{proposition}\label{pro:maiaqaq1qaqks8}
Let  $\Phi$ be a growth function of both lower type $p$ and upper type $q$. If  $\Phi$ is convex and  $1< p \leq q < \infty$ then  $\Phi \in \Delta_{2} \cap \nabla_{2}$ and $p \leq a_{\Phi} \leq b_{\Phi}\leq q$. 
\end{proposition}

\begin{proof}
Since $\Phi$ is a convex growth function of lower type $p> 1$, we deduce that the function $t\mapsto \frac{\Phi(t)}{t}$ (resp. $t\mapsto \frac{\Phi(t)}{t^{p}}$)  is increasing (resp. quasi-increasing) on  $(0, \infty)$. For  $t> 0$, we have
$$ \int_{0}^{t} \frac{\Phi(s)}{s^{2}}ds = \int_{0}^{t} \frac{\Phi(s)}{s^{p}}\times s^{p-2} ds \leq C \frac{\Phi(t)}{t^{p}} \int_{0}^{t}  s^{p-2} ds \leq  C'\frac{\Phi(t)}{t} $$
and  
$$ \int_{0}^{t} \frac{\Phi(s)}{s^{2}}ds \geq \int_{t/d}^{t} \frac{\Phi(s)}{s}\times \frac{1}{s} ds \geq \frac{\Phi(t/d)}{t/d} \int_{t/d}^{t}\frac{1}{s} ds = \frac{\Phi(t/d)}{t/d} \ln (d),  $$
for all  $d> 1$. We deduce that
$$  \Phi(t/d) \leq \frac{C'}{d \ln (d)} \Phi(t).  $$
By choosing $d=\exp( 2C')$, we obtain
$$  \Phi(t/d) \leq \frac{1}{2d } \Phi(t).  $$
Therefore $\Phi$ satisfies $\nabla_{2}-$condition. Since  $\Phi$ is also of upper type $q$, we deduce that $\Phi$ satisfies the $\Delta_{2}-$condition. Indeed, for  $t> 0$,
$$  \Phi(2t) \lesssim 2^{q}\Phi(t). $$
It follows that,  $\Phi \in \Delta_{2} \cap \nabla_{2}$.

\medskip

Let $t \geq 0$ and  $0 <s <1$ (resp.  $s> 1$), we have respectively
$$ \Phi(s) = \Phi(s\times 1) \leq C s^{p} \Phi(1) \Rightarrow \frac{\Phi(s)}{s} \lesssim s^{p-1} $$
and 
$$ \Phi(1) = \Phi\left(\frac{1}{s}\times s\right) \leq C \left(\frac{1}{s}\right)^{p} \Phi(s) \Rightarrow  s^{p-1}\lesssim  \frac{\Phi(s)}{s}. $$
We deduce that,  $\lim_{s \to 0}\frac{\Phi(s)}{s}=0$ and $\lim_{s \to \infty}\frac{\Phi(s)}{s}=+\infty$. Therefore,  there exists  $\varphi$ a positive, left continuous and non-decreasing function on $[0, \infty)$ such that $\varphi(0)=0$ and $\lim_{t \to \infty}\varphi(t)=\infty$ and, 
$$   \Phi(t)=\int_{0}^{t}\varphi(s)ds, ~~\forall~t\geq 0.
  $$
Since  $\Phi \in \Delta_{2} \cap \nabla_{2}$, we deduce that 
 $1< a_\Phi \leq b_\Phi< \infty$.

Let's show that  $p \leq a_{\Phi}$.  For  $t> 0$, we have
$$\Phi(t) \leq \Phi(2t) = \int_{0}^{2t}\varphi(s)ds \leq \int_{0}^{t}\frac{\Phi(s)}{s}ds \leq   \frac{\Phi(t)}{t^{p}} \int_{0}^{t}  s^{p-1} ds =\frac{1}{p} \Phi(t) \leq \frac{1}{p} t\varphi(t). $$
We deduce that
$  p \leq  \frac{ t\varphi(t)}{\Phi(t)}.   $

To show that   $b_{\Phi}\leq q$, we reason absurdly by supposing that    $b_{\Phi}> q$. Let $\Psi$ be the complementary function of $\Phi$. Since 
 $\Phi \in \Delta_{2}\cap \nabla_{2}$, we deduce that  $\Psi \in \Delta_{2}\cap \nabla_{2}$. Consequently,  $1< a_\Psi \leq b_\Psi< \infty$. Moreover, $\frac{1}{a_\Psi}  + \frac{1}{b_\Phi} =1$. We have, $$ \frac{1}{a_\Psi}+ \frac{1}{q} >  \frac{1}{a_\Psi}  + \frac{1}{b_\Phi} =1,     $$
Which is absurd because  $q> 1$  and  $a_{\Psi} > 1$.
\end{proof}

A growth function  $\Phi$ is equivalent to a convex function if and only if there exists $c > 1$ such that for all  $0<t_{1}<t_{2}$,
\begin{equation}\label{eq:suia8n}
\frac{\Phi(t_{1})}{t_{1}} \leq c\frac{\Phi(ct_{2})}{t_{2}},
 \end{equation}
(see  \cite[Lemma 1.1.1]{kokokrbec}).

\begin{lemma}\label{pro:maiaqaq1q8}
Let  $\Phi$ be a growth function of lower type $p\in (0, \infty)$. The following assertions are satisfied: 
\begin{itemize}
\item[(i)]  $\Phi$ is equivalent to a continuous and increasing  growth function of lower type $p$.
\item[(ii)] The growth function $\Phi_{p}$ defined by
\begin{equation}\label{eq:suiaqaq8n}
\Phi_{p}(t)=\Phi\left( t^{1/p}\right),~~\forall~t\geq 0\end{equation}
is equivalent to a continuous, increasing and convex growth function.
\end{itemize}
\end{lemma}

\begin{proof}
Since $\Phi$ is a growth function of lower type $p$, we have
$$   \Phi(st) \leq cs^{p}\Phi(t), ~~\forall~0<s\leq 1, ~~\forall~t\geq 0, $$
where $c$  is a constant that depends only on $p$. We deduce that $$ \frac{\Phi_{p}(t_{1})}{t_{1}} \leq c \frac{\Phi_{p}(t_{2})}{t_{2}} ,~~\forall~0<t_{1}< t_{2},   $$
where $\Phi_{p}$ is the growth function defined in (\ref{eq:suiaqaq8n}). It follows that $\Phi_{p}$ is equivalent to a convex function on $[0, \infty)$. 
In particular,  $\Phi$ and  $\Phi_{p}$ are equivalent to continuous and increasing functions on  $[0, \infty)$.
\end{proof}

\subsection{Hardy-Orlicz Space on  $\mathbb{D}$.} 

Let $\Phi_{1}$ and $\Phi_{2}$ be two growth functions of lower type. If $\Phi_{1}$ is equivalent to  $\Phi_{2}$ (i.e: $\Phi_{1} \sim \Phi_{2}$) then  $L^{\Phi_{1}}=L^{\Phi_{2}}$ (resp. $H^{\Phi_{1}}=H^{\Phi_{2}}$). Moreover, for any function $f$ belongs to $L^{\Phi_{1}}$ (resp. $H^{\Phi_{1}}$), we have $\|f\|_{L^{\Phi_{1}}}^{lux} \approx \|f\|_{L^{\Phi_{2}}}^{lux}$ (resp. $\|f\|_{H^{\Phi_{1}}}^{lux} \approx \|f\|_{H^{\Phi_{2}}}^{lux}$). We can therefore, without loss of generality, replace the equivalence condition between  $\Phi_{1}$ and  $\Phi_{2}$ with an equality, (i.e:  $\Phi_{1}=\Phi_{2}$), when we work in Orlicz type spaces.

\begin{remark}\label{pro:main 5aqaqq2pl}
In the following, we will assume that any growth function 
$\Phi$ of lower type $p \in (0, \infty)$ is continuous and increasing. Furthermore, the function  $\Phi_{p}$, defined in Relation (\ref{eq:suiaqaq8n}), is continuous and convex  growth function, thanks to Lemma \ref{pro:maiaqaq1q8}.
\end{remark}

Let $G : \mathbb{D} \longrightarrow \mathbb{R}$ be a function. We say that $G$ is a subharmonic function on $\mathbb{D}$ if $G$ is continuous on $\mathbb{D}$ and for all $z_{0} \in \mathbb{D}$, there exists $\rho_{0}>0$ such that $\mathcal{D}(z_{0}, \rho_{0}):=\{ \omega \in \mathbb{D}: |\omega-z_{0}| < \rho_{0} \} \subset \mathbb{D}$ with more
$$ G(z_{0})  \leq \frac{1}{2\pi}\int_{0}^{2\pi}G(z_{0}+\rho e^{it})dt, ~ ~ \forall~\rho<\rho_{0}.  $$
We have the following assertions (see \cite{Pduren2, Jbgarnett, javadmas}):
\begin{itemize}
\item[(i)] For $0<p<\infty$ and for any holomorphic function $G$ on  $\mathbb{D}$, the function $|G|^{p}$ is a subharmonic  on $\mathbb{D}$.
\item[(ii)]  Let  $\Phi$ be a increasing, continuous and convex function on  $(-\infty, \infty)$ and $G$ a subharmonic function on  $\mathbb{D}$. Then  the function  $\Phi(G)$ is a subharmonic  on $\mathbb{D}$.
\item[(iii)] For any subharmonic function  $G$ on  $\mathbb{D}$, the function $r\mapsto \frac{1}{2\pi}\int_{0}^{2\pi}|G(re^{it})|dt$ is non-decreasing on $[0,1[$.
\end{itemize}

Let $\Phi$ be a growth function of lower type $p \in (0, \infty)$. Since  $\Phi$ is continuous and increasing (see Remark \ref{pro:main 5aqaqq2pl}), we have the following continuous inclusions:
\begin{equation}\label{eq:suqiaqaqsq8n}
L^{\infty}(\mathbb{T}) \hookrightarrow L^{\Phi}(\mathbb{T})  \hookrightarrow L^{p}(\mathbb{T}) 
\end{equation}
and
\begin{equation}\label{eq:suqiaqsq8n}
H^{\infty}(\mathbb{D}) \hookrightarrow H^{\Phi}(\mathbb{D})  \hookrightarrow H^{p}(\mathbb{D}). 
\end{equation}

\begin{lemma}\label{pro:main5Qpaqm5}
 Let   $\Phi$ be a growth function of lower type. An analytic function $G$ on  $\mathbb{D}$ belongs to $H^{\Phi}(\mathbb{D})$ if and only if 
  $\lim_{r \to 1}\frac{1}{2\pi}\int_{0}^{2\pi}\Phi(|G(re^{it})|)dt< \infty.
  $ Moreover, $\|G\|_{H^{\Phi}}^{lux} = \lim_{r \to 1}\|G_{r}\|_{L^{\Phi}}^{lux},$ where $G_{r}$ is the function 
  define in (\ref{eq:equivaaqlent}). 
 \end{lemma}

 \begin{proof}
Suppose that $\Phi$ is of lower type $p\in (0, \infty)$. 
Since   $\Phi_{p}$ is a increasing, continuous and convex growth function (see Remark \ref{pro:main 5aqaqq2pl}) and $|G|^{p}$ is  subharmonic function on $\mathbb{D}$, we deduce that  $\Phi_{p}\left( |G|^{p}\right)$ is subharmonic on $\mathbb{D}$. It follows  that, the function $r\mapsto \frac{1}{2\pi}\int_{0}^{2\pi}\Phi_{p}(|G(re^{it})|^{p})dt$ is non-decreasing on $[0,1[$. We conclude that,  $G \in H^{\Phi}(\mathbb{D})$ if and only if 
   $\lim_{r \to 1}\frac{1}{2\pi}\int_{0}^{2\pi}\Phi(|G(re^{it})|)dt< \infty.$ Moreover,  $\|G\|_{H^{\Phi}}^{lux} = \lim_{r \to 1}\|G_{r}\|_{L^{\Phi}}^{lux}$.
  \end{proof}

\begin{lemma}\label{pro:mainfaaqaaqqq5}
Let  $0< s  < \infty$ and   $\Phi$ a growth function of lower type. If  $G\in H^{s}(\mathbb{D})$  and  if the function 
$g$ is defined by, for all almost $\theta \in \mathbb{R}$, 
$$   g(e^{i\theta})=\lim_{r\to 1}G(re^{i\theta}),  $$
belongs to $L^{\Phi}\left(\mathbb{T}\right)$ then $G\in H^{\Phi}(\mathbb{D})$. Moreover,  $\|G\|_{H^{\Phi}}^{lux}= \|g\|_{L^{\Phi}}^{lux}$. 
\end{lemma}

\begin{proof}
Suppose that $\Phi$ is of lower type $p\in (0, \infty)$ 
and   $G \not\equiv 0$. Since $G\in H^{s}(\mathbb{D})$,   there exists a unique function
    $g\in L^{s}\left(\mathbb{T}\right)$ such that $\log|g| \in L^{1}\left(\mathbb{T}\right)$ and
 $ g(e^{i\theta})=\lim_{r\to 1}G(re^{i\theta}),$
 for almost all $\theta\in \mathbb{R}$. Moreover, $g(e^{it})\not=0$, for almost all $t\in \mathbb{R}$ and
$$  \log|G(re^{i\theta})| \leq   \frac{1}{2\pi}\int_{-\pi}^{\pi}P_{r}(e^{i(\theta-t)})\log|g(e^{it})|dt, ~~\forall~re^{i\theta} \in \mathbb{D},   $$
where $P_{r}$ is the Poisson kernel on $\mathbb{T}$, (see  \cite{javadmas}).  For $t\geq 0$, put
$$  \widetilde{\Phi}(t)= \Phi_{p}(\exp(t)).  $$
By construction, $\widetilde{\Phi}$ is a convex function on  $(-\infty, \infty)$ as a composition of two convex functions. Since  $g \in L^{\Phi}\left(\mathbb{T}\right)$, we can take $\|g\|_{L^{\Phi}}^{lux}=1$. For $0\leq r<1$ ,  we have
\begin{align*}
\frac{1}{2\pi}\int_{-\pi}^{\pi}\Phi\left( |G(re^{i\theta})|\right)d\theta &= \frac{1}{2\pi}\int_{-\pi}^{\pi}\widetilde{\Phi}\left( \log|G(re^{i\theta})|\right)d\theta \\ 
&\leq \frac{1}{2\pi}\int_{-\pi}^{\pi}\left( \frac{1}{2\pi}\int_{-\pi}^{\pi}P_{r}(e^{i(\theta-t)})d\theta     \right)\widetilde{\Phi}\left(\log|g(e^{it})| \right)dt \\
&=\frac{1}{2\pi}\int_{-\pi}^{\pi}\Phi\left( |g(e^{it})| \right)dt  
\leq 1,
\end{align*}
thanks to Jensen's inequality and Fubini's theorem.
We deduce that $G\in H^{\Phi}(\mathbb{D})$ and $\|G\|_{H^{\Phi}}^{lux} \leq\|g\|_{L^{\Phi}}^{lux}$.  The reverse  is obtained using Fatou's lemma. Indeed,
\begin{align*}
\frac{1}{2\pi}\int_{0}^{2\pi}\Phi\left( \frac{|g(e^{it})|}{\|G\|_{H^{\Phi}}^{lux}}  \right)dt 
&\leq
\liminf_{r \to 1}\frac{1}{2\pi}\int_{0}^{2\pi}\Phi\left( \frac{|G(re^{it})|}{\|G\|_{H^{\Phi}}^{lux}} \right) dt \\ &\leq   \sup_{0\leq r<1}\frac{1}{2\pi}\int_{0}^{2\pi}\Phi\left( \frac{|G(re^{it})|}{\|G\|_{H^{\Phi}}^{lux}}\right)dt \leq 1. 
\end{align*}
\end{proof}

The following result is an immediate consequence of Lemma \ref{pro:main5Qpaqm5} and Lemma \ref{pro:mainfaaqaaqqq5}. Therefore, the proof will be omitted.

\begin{theorem}\label{pro:mainfaaqaqq5}
Let   $\Phi$ be a growth function of lower type. For $0\not\equiv G\in H^{\Phi}(\mathbb{D})$, there exists a unique function
    $g\in L^{\Phi}\left(\mathbb{T}\right)$ such that $\log|g| \in L^{1}\left(\mathbb{T}\right)$, 
 $ g(e^{i\theta})=\lim_{r\to 1}G(re^{i\theta}), $
 for almost all $\theta\in \mathbb{R}$, $g(e^{it})\not=0$, for almost all $t\in \mathbb{R}$ and 
$$ \log|G(re^{i\theta})| \leq   \frac{1}{2\pi}\int_{-\pi}^{\pi}P_{r}(e^{i(\theta-t)})\log|g(e^{it})|dt, ~~\forall~re^{i\theta} \in \mathbb{D}.
   $$
Moreover, 
\begin{equation}\label{eq:suqiajlqsqsq8n}
\|G\|_{H^{\Phi}}^{lux} =\lim_{r \to 1}\|G_{r}\|_{L^{\Phi}}^{lux}=\|g\|_{L^{\Phi}}^{lux},
\end{equation}
where $G_{r}$ is the function 
define in (\ref{eq:equivaaqlent}). 
\end{theorem}

Let $f$ be a measurable function on  $\mathbb{T}$. The maximal Hardy-Littlewood function, $\mathcal{M}_{HL}(f)$ of $f$  is defined by
$$  \mathcal{M}_{HL}(f)(u):=\sup_{I \subset \mathbb{T} }\frac{1}{|I|}\int_{I}| f(t)| dt, ~~\forall~ u \in \mathbb{T},  $$
where the supremum is taken over all intervals of $\mathbb{T}$. If  $\Phi \in \Delta_{2} \cap \nabla_{2}$ a convex growth function then  $\mathcal{M}_{HL}$  is defined on $L^{\Phi}(\mathbb{T}) \longrightarrow L^{\Phi}(\mathbb{T})$ and is bounded (see \cite[ Theorem 7]{rao68ren}).

Let $\alpha > 0$ and $G$ be a holomorphic function on $\mathbb{D}$.  The  maximal nontangential function $G_{\alpha}^{*}$ of $G$ is defined by
$$ G_{\alpha}^{*}(\zeta) =\sup_{z \in \Gamma_{\alpha}(\zeta)}|G(z)|, ~~\forall~\zeta \in \mathbb{T}, 
   $$
where $\Gamma_{\alpha}(\zeta):=\{z=r\omega \in \mathbb{D}: |\omega-\zeta| <\alpha(1-r)  \}$.

\begin{lemma}\label{pro:mainfaaaqqaq5}
Let  $\alpha > 0$ and  $\Phi$ a growth function of both lower type and upper type. Then  for any $G\in H^{\Phi}(\mathbb{D})$, the  maximal nontangential function
 $G_{\alpha}^{*}$ belongs to $L^{\Phi}(\mathbb{T})$. Moreover, \begin{equation}\label{eq:suqiaqapmqsqaq8n}
\|G_{\alpha}^{*}\|_{L^{\Phi}}^{lux} \approx \|G\|_{H^{\Phi}}^{lux}.
\end{equation}
\end{lemma}

\begin{proof}
Suppose that $\Phi$ is respectively of lower type $p$ and of upper type $q$. Put 
$$  \widetilde{\Phi}(t) =\Phi_{p}(t^{2}), ~~\forall~t \geq 0.  $$
By construction,  $\widetilde{\Phi}$ is a convex growth function as a composite of two convex growth functions. Furthermore,  $\widetilde{\Phi}$ is both lower type $p'=2$ and upper type $q'= \frac{2q}{p}$. We therefore deduce that  $\widetilde{\Phi} \in \Delta_{2} \cap \nabla_{2}$, thanks to Proposition \ref{pro:maiaqaq1qaqks8}. It follows that,  $\mathcal{M}_{HL}$  is defined on $L^{\widetilde{\Phi}}(\mathbb{T}) \longrightarrow L^{\widetilde{\Phi}}(\mathbb{T})$ and is bounded. 

\medskip
 
Let  $ G\in H^{\Phi}(\mathbb{D})$. Assume  $G \not\equiv 0$ because there is nothing to show when  $G \equiv 0$. According to Theorem \ref{pro:mainfaaqaqq5}, there exists a unique function
    $g\in L^{\Phi}\left(\mathbb{T}\right)$ such that $\log|g| \in L^{1}\left(\mathbb{T}\right)$, 
 $ g(e^{i\theta})=\lim_{r\to 1}G(re^{i\theta}), $
 for almost all $\theta\in \mathbb{R}$, $g(e^{it})\not=0$, for almost all $t\in \mathbb{R}$ and 
\begin{equation}\label{eq:suqiajsqpmsq8n}
\log|G(re^{i\theta})| \leq   \frac{1}{2\pi}\int_{-\pi}^{\pi}P_{r}(e^{i(\theta-t)})\log|g(e^{it})|dt, ~~\forall~re^{i\theta} \in \mathbb{D}.
\end{equation}
Moreover, $\|G\|_{H^{\Phi}}^{lux} = \|g\|_{L^{\Phi}}^{lux}$. From Relation (\ref{eq:suqiajsqpmsq8n}), we deduce that for  $\alpha > 0$ and $e^{i\theta} \in \mathbb{T}$, 
$$  G_{\alpha}^{*}(e^{i\theta}) \lesssim \left(\mathcal{M}_{HL}(|g|^{p/2})(e^{i\theta})\right)^{2/p}. $$
Since  $g\in L^{\Phi}\left(\mathbb{T}\right)$, we deduce that $|g|^{p/2} \in L^{\widetilde{\Phi}}(\mathbb{T})$ and $\||g|^{p/2}\|_{L^{\widetilde{\Phi}}}^{lux}=\|g\|_{L^{\Phi}}^{lux}$. It follows that, 
$$  \frac{1}{2\pi}\int_{-\pi}^{\pi}\Phi\left( \frac{|G_{\alpha}^{*}(e^{i\theta})|}{\|G\|_{H^{\Phi}}^{lux}}\right)d\theta \lesssim \frac{1}{2\pi}\int_{-\pi}^{\pi}\widetilde{\Phi}\left( \frac{|\mathcal{M}_{HL}(|g|^{p/2})(e^{i\theta})|}{\||g|^{p/2}\|_{L^{\widetilde{\Phi}}}^{lux}}\right)d\theta \lesssim 1.    $$
Therefore, $G_{\alpha}^{*} \in L^{\Phi}(\mathbb{T})$ and $\|G_{\alpha}^{*}\|_{L^{\Phi}}^{lux} \lesssim \|G\|_{H^{\Phi}}^{lux}$. The inverse inequality is obvious, since
$$ |G(re^{it})| \leq  G_{\alpha}^{*}(e^{it}),   $$
for all $t \in \mathbb{R}$.
\end{proof}

\begin{proposition}\label{pro:mainfaaaaqaqlaq5}
Let  $\Phi$ be a growth function of both lower type and upper type. For all  $G\in H^{\Phi}(\mathbb{D})$,  there exists a unique function  $g\in L^{\Phi}\left(\mathbb{T}\right)$ such that
 $ g(e^{i\theta})=\lim_{r\to 1}G(re^{i\theta}), $
 for almost all $\theta\in \mathbb{R}$ and  $\|G\|_{H^{\Phi}}^{lux} =\|g\|_{L^{\Phi}}^{lux}$. Moreover, 
\begin{equation}\label{eq:suqiaqapmqqasqaq8n}
\lim_{r\to 1}\|G_{r}-g\|_{L^{\Phi}}^{lux} =0,
\end{equation}
where $G_{r}$ is the function 
define in (\ref{eq:equivaaqlent}). 
\end{proposition}

\begin{proof}
Let  $G\in H^{\Phi}(\mathbb{D})$. Assume  $G \not\equiv 0$ because there is nothing to show when  $G \equiv 0$. According to Theorem \ref{pro:mainfaaqaqq5}, there exists a unique function
    $g\in L^{\Phi}\left(\mathbb{T}\right)$ such that
 $ g(e^{i\theta})=\lim_{r\to 1}G(re^{i\theta}), $
 for almost all $\theta\in \mathbb{R}$ and  $\|G\|_{H^{\Phi}}^{lux} =\|g\|_{L^{\Phi}}^{lux}$.  
To prove the equality of Relation (\ref{eq:suqiaqapmqqasqaq8n}), it suffices to show that 
 $$ \lim_{r \to 1}\frac{1}{2\pi}\int_{0}^{2\pi}\Phi(\varepsilon|G(re^{it})-g(e^{it})|)dt=0, ~~\forall~\varepsilon > 0.    $$ 
The function $G_{1}^{*}$ belongs to $L^{\Phi}(\mathbb{T})$, thanks to Theorem \ref{pro:mainfaaaqqaq5}. Moreover, for  $0 \leq r <1$, we have
  $$  \Phi\left(\left|G_{r}(e^{i\theta})-g(e^{i\theta})\right|\right) \lesssim  \Phi\left(\sup_{0 \leq r <1}|G(re^{i\theta})|\right) \lesssim \Phi(G_{1}^{*}(e^{i\theta})),   $$
  for almost all $\theta \in \mathbb{R}$.  According to theorem of dominated convergence, it follows that for  $\varepsilon > 0$, 
   $$ \lim_{r\to 1}\frac{1}{2\pi}\int_{0}^{2\pi}\Phi(\varepsilon|G(re^{i\theta})-g(e^{i\theta})|)d\theta= \frac{1}{2\pi}\int_{0}^{2\pi}\lim_{r\to 1}\Phi(\varepsilon|G(re^{i\theta})-g(e^{i\theta})|)d\theta =0.  $$
\end{proof}

For  $g \in L^{1}(\mathbb{T})$,  the $n$th Fourier coefficient of $g$ is defined by
$$  \widehat{g}(n) =\frac{1}{2\pi}\int_{-\pi}^{\pi}g(e^{it})e^{-int}dt, ~~(n \in \mathbb{Z}).  $$

\medskip

For a convex growth function  $\Phi$, we define, we define $ H^{\Phi}(\mathbb{T})$ as
\begin{equation}\label{eq:5aqqxaqaaqq656}
H^{\Phi}(\mathbb{T})=\{ g \in L^{\Phi}(\mathbb{T}): \widehat{g}(n)=0, ~~\forall~n <0   \}.
\end{equation}
We know that   $L^{\Phi}(\mathbb{T})$ is Banach space and as $H^{\Phi}(\mathbb{T})$ is a closed subspace of $L^{\Phi}(\mathbb{T})$. Thus $H^{\Phi}(\mathbb{T})$ is also a Banach space for the induced norm of $L^{\Phi}(\mathbb{T})$. 

\medskip

In this work, we refrain from defining  $ H^{\Phi}(\mathbb{T})$ when $\Phi$ is concave, as this definition will not be necessary for our purposes. For readers wishing to deepen their understanding of this case, we refer them to \cite{BoGre}.

The following result also follows from  $H^{\Phi}(\mathbb{D})$ is contained in $H^{1}(\mathbb{D})$, when $\Phi$ is a growth function of lower type $p\geq 1$ and   \cite[Theorem 5.11 and Corollary 5.12]{javadmas}).

\begin{lemma}\label{pro:main11pqapaqmpmm6}
Let $\Phi$ be a growth function of lower type $p\geq 1$ and  $G$ an analytic function on $\mathbb{D}$. Then  $G \in H^{\Phi}(\mathbb{D})$ if and only if there exists a unique function  $g \in H^{\Phi}(\mathbb{T})$  such that
$$ G(z)=   \frac{1}{2\pi}\int_{-\pi}^{\pi}P_{r}(\theta-t)g(e^{it})dt= \sum_{n=0}^{\infty}\widehat{g}(n)z^{n},
   $$
for all $z=re^{i\theta} \in \mathbb{D}$. The series is uniformly convergent on compact subsets of $\mathbb{D}$. Moreover, 
$$ G(z)= \frac{1}{2\pi}\int_{-\pi}^{\pi}\frac{g(e^{it})}{1-e^{-it}z}dt
  $$
and 
$$  \int_{-\pi}^{\pi}\frac{\overline{z}e^{it}}{\overline{z}e^{it}-1}g(e^{it})dt= 0,
  $$
for all  $z \in \mathbb{D}$ and, 
\begin{equation}\label{eq:suqiajlqsqsq8n}
\|G\|_{H^{\Phi}}^{lux} =\|g\|_{L^{\Phi}}^{lux}.
\end{equation}
\end{lemma}

Let $u \in L^{1}(\mathbb{T})$. The Hilbert transform $\mathcal{H}(u)$ of $u$ at the point $e^{i\theta}\in \mathbb{T}$ is defined by
$$  \mathcal{H}(u)(e^{i\theta})=\lim_{\varepsilon \to 0}\frac{1}{\pi}\int_{\varepsilon<|t| <\pi }\frac{u(e^{i(\theta-t)})}{2\tan(t/2)}dt,  $$
wherever the limit exists. If  $\Phi \in \Delta_{2} \cap \nabla_{2}$ a convex growth function then the Hilbert transform $\mathcal{H}$  is defined on $L^{\Phi}(\mathbb{T}) \longrightarrow L^{\Phi}(\mathbb{T})$ and is bounded (see \cite{rao68ren}).

\begin{lemma}\label{pro:main11pqaplas6}
Let  $\Phi$ be a growth function of both lower type $p$ and upper type $q$. Let $u$ be a real function defined on 
$\mathbb{T}$ and belonging to $L^{\Phi}(\mathbb{T})$ and, let
\begin{equation}\label{eq:suqiajlqaqaqsqsq8n}
G(\omega)= \frac{1}{2\pi}\int_{\mathbb{T}}\frac{e^{it}+\omega}{e^{it}-\omega}u(e^{it})dt, ~~\forall~\omega \in \mathbb{D}.
\end{equation}
If $1< p \leq q< \infty$ then $G \in H^{\Phi}(\mathbb{D})$ and
\begin{equation}\label{eq:suqiajlqaqsqsq8n}
\|u\|_{L^{\Phi}}^{lux} \leq \|G\|_{H^{\Phi}}^{lux} \leq C_{\Phi}\|u\|_{L^{\Phi}}^{lux}, 
\end{equation}
where $C_{\Phi}$ is a constant just depending on $\Phi$. Moreover, the boundary values of $G$ are given by
$$ g:= u+i \mathcal{H}(u) \in H^{\Phi}(\mathbb{T}).    $$
\end{lemma}

\begin{proof}
For $t \geq 0$, we have
$ \Phi(t) =\Phi_{p}(t^{p}).   $
We deduce that $\Phi$ is a convex growth function as a composite of two convex growth functions. Moreover,  $\Phi \in \Delta_{2} \cap \nabla_{2}$, thanks to Proposition \ref{pro:maiaqaq1qaqks8}. It follows that, the Hilbert transform $\mathcal{H}$  is defined on $L^{\Phi}(\mathbb{T}) \longrightarrow L^{\Phi}(\mathbb{T})$ and is bounded.

\medskip

Since  $u \in L^{p}(\mathbb{T})$ (see Relation \ref{eq:suqiaqaqsq8n}), we deduce that the function $G$ belongs to $H^{p}(\mathbb{D})$ and
$$  g(e^{it}):=\lim_{r \to 1}G(re^{it}) = u(e^{it})+i \mathcal{H}(u)(e^{it}), 
  $$
for almost all $t \in  \mathbb{R}$ (see \cite[Corollary 6.7]{javadmas}).  We have
$$  \|g\|_{L^{\Phi}}^{lux} \lesssim \|u\|_{L^{\Phi}}^{lux} + \|\mathcal{H}(u)\|_{L^{\Phi}}^{lux} \lesssim \|u\|_{L^{\Phi}}^{lux}< \infty.  $$ 
We deduce that  $g \in L^{\Phi}(\mathbb{T})$. It follows that $G \in H^{\Phi}(\mathbb{D})$ and $\|G\|_{H^{\Phi}}^{lux} =\|g\|_{L^{\Phi}}^{lux}$, according to Lemma \ref{pro:mainfaaqaaqqq5}. Since $|u(\omega)| \leq |g(\omega)|$, for almost all $\omega \in  \mathbb{T}$, we have
 $\|u\|_{L^{\Phi}}^{lux}\leq \|g\|_{L^{\Phi}}^{lux}$. Moreover,  $\widehat{g}(n)=0$, for $n\leq 0$, according to Lemma \ref{pro:main11pqapaqmpmm6}.
\end{proof}

The following result is an immediate consequence of Lemma \ref{pro:main11pqapaqmpmm6}  and Lemma \ref{pro:main11pqaplas6}.

\begin{proposition}\label{pro:main11pqapaqlas6}
Let  $\Phi \in \Delta_{2} \cap \nabla_{2}$ be a convex growth function.  The Szeg\"o projection $ \mathcal{P}$  maps  $L^{\Phi}(\mathbb{T})$ boundedly onto  $H^{\Phi}(\mathbb{D})$.
\end{proposition}

\section{Proof of the main results.} 

\subsection{Factorization of Hardy-Orlicz Spaces.} 

\subsubsection{Proof of Theorem \ref{pro:main2a}.} 

\begin{proposition}\label{pro:main5Qpaqqm5}
Let   $\Phi$ be a growth function of lower type. For  $0\not\equiv G\in H^{\Phi}(\mathbb{D})$, if $\{\omega_{n}\}_{n\in \mathbb{N}}$ is the sequence of zeros $G$ in $\mathbb{D}$ then  
$  \sum_{n\geq 0} (1- |\omega_{n}|) < \infty. $
\end{proposition}

\begin{proof}
Since  $H^{\Phi}(\mathbb{D})$ is a subset of $H^{p}(\mathbb{D})$, for some $p\in (0, \infty)$, thanks to Relation (\ref{eq:suqiaqsq8n}). Thus, we have $  \sum_{n\geq 0} (1- |\omega_{n}|) < \infty $, thanks to \cite[Lemma 7.6]{javadmas}.
\end{proof}

It follows from Proposition \ref{pro:main5Qpaqqm5} that for any function $G$ belonging to $H^{\Phi}(\mathbb{D})$, the Blaschke product associated with the sequence of zeros of $G$ is well-defined.

\proof[Proof of Theorem \ref{pro:main2a}.]
Let  $0\not\equiv G\in H^{\Phi}(\mathbb{D})$ and let  $(z_{n})_{n\geq 1}$ be the sequence zeros of $G$ on $\mathbb{D}$. According to Theorem \ref{pro:mainfaaqaqq5}, there exists a unique function  $g\in L^{\Phi}\left(\mathbb{T}\right)$ such that 
 $ g(e^{i\theta})=\lim_{r\to 1}G(re^{i\theta}), $
 for almost all $\theta\in \mathbb{R}$ and  $\|G\|_{H^{\Phi}}^{lux} =\|g\|_{L^{\Phi}}^{lux}$. The Blaschke product $B$ associated with the sequence zeros  $(z_{n})_{n\geq 1}$ is well defined, according to Proposition \ref{pro:main5Qpaqqm5}. As  $B$ is an inner function, we deduce that $B\in H^{\infty}(\mathbb{D})$ 
and there exists  $b\in L^{\infty}(\mathbb{T})$ such that
$b(e^{it})=\lim_{r\to 1}B(re^{it})$ and  $|b(e^{it})|=1,$
for almost all $t\in \mathbb{R}$. It follows that, $g/b \in L^{\Phi}(\mathbb{T})$ and  $\|g/b\|_{L^{\Phi}}^{lux}=\|g\|_{L^{\Phi}}^{lux}$. Since  $G\in H^{p}(\mathbb{D})$, for some  $p\in (0, \infty)$ and   $B$ the Blaschke product associated with the sequence zeros  $(z_{n})_{n\geq 1}$ of $G$, we deduce that  $G/B\in H^{p}(\mathbb{D})$. Moreover, 
$$  \lim_{r\to 1} \frac{G(re^{i\theta})}{B(re^{i\theta})}= \frac{g(e^{i\theta})}{b(e^{i\theta})},  $$
for almost all $\theta\in \mathbb{R}$.  We deduce that $G/B\in H^{\Phi}(\mathbb{D})$ and
$$  \|G/B\|_{H^{\Phi}}^{lux}  =  \|g/b\|_{L^{\Phi}}^{lux}=\|g\|_{L^{\Phi}}^{lux}  =  \|G\|_{H^{\Phi}}^{lux},    $$
according to Lemma \ref{pro:mainfaaqaaqqq5}. 
\epf

\subsubsection{Proof of Theorem \ref{pro:main0apaaal0}.} 

The following two results directly follow from  Lemma \ref{pro:mainfaaqaaqqq5}. Consequently, the proofs will be omitted.

\begin{proposition}\label{pro:mainfpmqaaq5}
Let   $\Phi$ a growth function of lower type and $g$  a measurable function on $\mathbb{T}$ such that $\log|g| \in L^{1}(\mathbb{T})$. Then  $g \in L^{\Phi}(\mathbb{T})$ if and only if $O_{|g|} \in H^{\Phi}(\mathbb{D})$, where  $O_{|g|}$ is the outer function associated with $|g|$. Moreover, $\|O_{|g|}\|_{H^{\Phi}}^{lux}=\|g\|_{L^{\Phi}}^{lux}$.
\end{proposition}

\begin{proposition}\label{pro:main2mpaqq}
Let   $\Phi$ a growth function of lower type and   $G$ 
an inner function on $\mathbb{D}$. If $F \in H^{\Phi}(\mathbb{D})$ then the product  $GF \in H^{\Phi}(\mathbb{D})$ and $\|GF\|_{H^{\Phi}}^{lux} =\|F\|_{H^{\Phi}}^{lux}$.
\end{proposition}

\proof[Proof of Theorem \ref{pro:main0apaaal0}.]
Let  $0\not\equiv G\in H^{\Phi}(\mathbb{D})$. Since $H^{\Phi}(\mathbb{D})$ is a subset of $H^{p}(\mathbb{D})$, for some  $p\in (0, \infty)$, we deduce that $G\in H^{p}(\mathbb{D})$. There is therefore  a unique function
    $g\in L^{p}\left(\mathbb{T}\right)$ such that $\log|g| \in L^{1}\left(\mathbb{T}\right)$ and
 $ g(e^{i\theta})=\lim_{r\to 1}G(re^{i\theta}),$
 for almost all $\theta\in \mathbb{R}$,  and there exists also
$\sigma$ a unique positive finite and singular Borel measure on $\mathbb{T}$ such that 
$$  G(z) = B(z)S_{\sigma}(z)O_{|g|}(z), ~~\forall~z\in \mathbb{D},   $$
where $B$ is the Blaschke product associated with the zero sequence of $G$,  $O_{|g|}$ is outer function associated with $|g|$ and $S_{\sigma}$ is the inner function  defined by 
\begin{equation}\label{eq:suqi8n}
S_{\sigma}(\omega)= \exp\left\{-\frac{1}{2\pi}\int_{-\pi}^{\pi}\frac{e^{it}+\omega}{e^{it}-\omega}d\sigma(e^{it})  \right\}, ~~\forall~\omega \in \mathbb{D},
\end{equation}
(see \cite{Jbgarnett}). 
As  $G\in H^{\Phi}(\mathbb{D})$, we deduce that   $g\in L^{\Phi}\left(\mathbb{T}\right)$, thanks to Fatou's Lemma. It follows that, $O_{|g|}\in H^{\Phi}(\mathbb{D})$ and $\left\|O_{|g|}\right\|_{H^{\Phi}}^{lux}=\|g\|_{L^{\Phi}}^{lux}$, according to Proposition \ref{pro:mainfpmqaaq5}. 
Since the Blaschke product $B$ and the function  $S_{\sigma}$  are inner functions on $\mathbb{D}$, we deduce that $B S_{\sigma}$  is inner function on $\mathbb{D}$. It follows that 
$$\|G\|_{H^{\Phi}}^{lux} =\left\|BS_{\sigma}O_{|g|}\right\|_{H^{\Phi}}^{lux} =\left\|O_{|g|}\right\|_{H^{\Phi}}^{lux},  $$
according to Proposition \ref{pro:main2mpaqq}.
\epf

\subsubsection{Proof of Theorem \ref{pro:main2aqop}.} 

Let  $p> 0$ and $\Phi$ be a   growth function. Then  $\Phi$ is of lower type $p$ if and only if $\Phi^{-1}$ is of upper type $1/p$  (see \cite{sehbaedgc1}).

\begin{lemma}\label{pro:mainfqmpq5}
Let  $\Phi_{1}$ and $\Phi_{2}$  be two growth functions of lower type  $p_{1}$ and $p_{2}$ respectively. Let  $\Phi_{3}$ be a positive function on $[0, \infty)$ such that $\Phi_{3}^{-1} = \Phi_{1}^{-1}.\Phi_{2}^{-1},$ where  $\Phi_{j}^{-1}$ is the inverse function of  $\Phi_{j}$, for  $j\in \{1,2,3\}$. 
Then the function $\Phi_{3}$ is a growth function of lower type $r:=\left( 1/p_{1}+1/p_{2}  \right)^{-1}$. 
If, moreover, one of the functions  $\Phi_{1}$ or  $\Phi_{2}$ is of upper type $q$ then $\Phi_{3}$ is also of upper type $q$. In this case we have $0<r\leq q <\infty$.
\end{lemma}

\begin{proof}
Since  $\Phi_{1}$ and  $\Phi_{2}$ are increasing homeomorphisms of  $[0, \infty)$ onto  $[0, \infty)$, we deduce that  $\Phi_{1}^{-1}$ and  $\Phi_{2}^{-1}$ are growth functions of upper type $1/p_{1}$ and $1/p_{2}$ respectively. It follows that,  $\Phi_{3}^{-1}$ is a bijective growth function of upper type $r=1/p_{1}+1/p_{2}$. Indeed, for all $t \geq 1$ and  $s >0$, we have
$$ \Phi_{3}^{-1}(st) =  \Phi_{1}^{-1}(st).\Phi_{2}^{-1}(st) \lesssim t^{1/p_{1}}t^{1/p_{2}}\Phi_{1}^{-1}(s).\Phi_{2}^{-1}(s)= t^{r}\Phi_{3}^{-1}(s). $$
Therefore, $\Phi_{3}$ is a growth function of lower type $1/r$.

\medskip

Suppose that  $\Phi_{1}$ is of upper type $q$. Since  $\Phi_{1}^{-1}$ is of lower type $1/q$ et  $\Phi_{2}^{-1}$ is increasing, we deduce that  $\Phi_{3}^{-1}$ is of lower type $1/q$. Indeed, for all $0<t <1$ and  $s >0$, we have
$$ \Phi_{3}^{-1}(st) =  \Phi_{1}^{-1}(st).\Phi_{2}^{-1}(st) \lesssim t^{1/q}\Phi_{1}^{-1}(s).\Phi_{2}^{-1}(s)= t^{1/q}\Phi_{3}^{-1}(s). $$
Therefore, $\Phi_{3}$ is a growth function of upper type $q$.
\end{proof}

Let us recall a result of Volberg and Tolokonnikov on Orlicz spaces in \cite[Lemma 3]{voltiko}. Their result can be reformulated as follows:

\begin{proposition}\label{pro:main 5aqaq2pl}
Let  $\Phi_{1}, \Phi_{2}$ and $\Phi_{3}$  be growth functions of the lower type. $L^{\Phi_{3}}(\mathbb{T})=L^{\Phi_{1}}(\mathbb{T}).L^{\Phi_{2}}(\mathbb{T})$ if and only if  $\Phi_{3}^{-1} \sim \Phi_{1}^{-1}.\Phi_{2}^{-1}$,  where  $\Phi_{j}^{-1}$ is the inverse function of  $\Phi_{j}$, for  $j\in \{1,2,3\}$. 
\end{proposition}

The following result is an immediate consequence of Proposition \ref{pro:main 5aqaq2pl}. Therefore, the proof will be omitted.

\begin{lemma}\label{pro:main 5aqp1aqqq2pl}
Let  $\Phi_{1}, \Phi_{2}$ and $\Phi_{3}$  be growth functions of the lower type such that  $\Phi_{3}^{-1} \sim \Phi_{1}^{-1}.\Phi_{2}^{-1}$,  where  $\Phi_{j}^{-1}$ is the inverse function of  $\Phi_{j}$, for  $j\in \{1,2,3\}$.  For all  $F\in   H^{\Phi_{1}}(\mathbb{D})$ and  $G\in   H^{\Phi_{2}}(\mathbb{D})$, the product   $FG \in H^{\Phi_{3}}(\mathbb{D})$ and  
\begin{equation}\label{eq:eaamqq1}
\|FG\|_{H^{\Phi_{3}}}^{lux} \leq C \|F\|_{H^{\Phi_{1}}}^{lux}\|G\|_{H^{\Phi_{2}}}^{lux}, 
\end{equation}
where $C$ is constant independent of $F$ and $G$.  
\end{lemma}

\proof[Proof of Theorem \ref{pro:main2aqop}.]
Without loss of generality, we can assume that  $\Phi_{3}^{-1}= \Phi_{1}^{-1}.\Phi_{2}^{-1}$. 

\medskip

Let  $0\not\equiv G\in H^{\Phi_{3}}(\mathbb{D})$.  According to Theorem \ref{pro:main0apaaal0}, there exists a unique function
    $g\in L^{\Phi_{3}}\left(\mathbb{T}\right)$ such that $\log|g| \in L^{1}\left(\mathbb{T}\right)$ and
 $ g(e^{i\theta})=\lim_{r\to 1}G(re^{i\theta}),$
 for almost all $\theta\in \mathbb{R}$ and there exists also
$\sigma$ a unique positive finite and singular Borel measure on $\mathbb{T}$ such that
$$   G(z) = B(z)S_{\sigma}(z)O_{|g|}(z), ~~\forall~z\in \mathbb{D}, 
    $$
where $B$ is the Blaschke product associated with the zero sequence of $G$,  $O_{|g|}$ is outer function associated with $|g|$ and $S_{\sigma}$ is the inner function  defined in (\ref{eq:suqi8n}). Moreover,   $\|G\|_{H^{\Phi_{3}}}^{lux}=\|O_{|g|}\|_{H^{\Phi_{3}}}^{lux} =\|g\|_{L^{\Phi_{3}}}^{lux}.$ 
For $k=\{1,2\}$, put   $$ g_{k}= \Phi_{k}^{-1}\circ\Phi_{3}\left(\|g\|_{L^{\Phi_{3}}}^{lux}\right)  \Phi_{k}^{-1}\circ\Phi_{3}\left(\frac{|g|}{\|g\|_{L^{\Phi_{3}}}^{lux}}\right).        $$

\medskip
In the rest of the proof, we can pose
$\|G\|_{H^{\Phi_{3}}}^{lux}=\|g\|_{L^{\Phi_{3}}}^{lux}=1$.

\medskip

By construction, $g_{k}\in  L^{\Phi_{k}}\left(\mathbb{T}\right)$ and
$$ \|g_{k}\|_{L^{\Phi_{k}}}^{lux} \leq \Phi_{k}^{-1}\circ\Phi_{3}\left(\|g\|_{L^{\Phi_{3}}}^{lux}\right).  $$
Indeed, 
$$  \frac{1}{2\pi}\int_{0}^{2\pi}\Phi_{k}\left( \frac{|g_{k}(e^{it})|}{\Phi_{k}^{-1}\circ\Phi_{3}(1)} \right)dt = \frac{1}{2\pi}\int_{0}^{2\pi}\Phi_{3}\left(|g(e^{it})|\right)
dt \leq 1.
   $$
Since  $\Phi_{3}^{-1}= \Phi_{1}^{-1}.\Phi_{2}^{-1}$, it follows that 
\begin{equation}\label{eq:inegaaqlitedehay}
\|g_{1}\|_{L^{\Phi_{1}}}^{lux}\|g_{2}\|_{L^{\Phi_{2}}}^{lux}  \leq \|g\|_{L^{\Phi_{3}}}^{lux}.
\end{equation}
We have also
 \begin{equation}\label{eq:inegaaqliaqtedehay}
|g| = g_{1}g_{2}.
\end{equation}
Indeed, 
\begin{align*}
g_{1}g_{2}&=\Phi_{1}^{-1}\circ\Phi_{3}(1)  \Phi_{1}^{-1}\circ\Phi_{3}\left(|g|\right) \Phi_{2}^{-1}\circ\Phi_{3}(1)\Phi_{2}^{-1}\circ\Phi_{3}\left(|g|\right) \\
&= \Phi_{3}^{-1}\circ\Phi_{3}(1) \Phi_{3}^{-1}\left(\Phi_{3}\left(|g|\right)\right)
=|g|. 
\end{align*}
Let us  assume that $\Phi_{1}$ and $\Phi_{2}$ are respectively of lower type $p_{1}$ and $p_{2}$. For  $k=\{1,2\}$, recall that the function $\Phi_{p_{k}}$ defined by  
 $$ \Phi_{p_{k}}(t) = \Phi_{k}(t^{1/p_{k}}), ~~\forall~ t \geq 0   $$ 
is continuous, increasing and convex (see Remark \ref{pro:main 5aqaqq2pl}). Moreover,
$$ \Phi_{p_{k}}^{-1}(t) = [\Phi_{k}^{-1}(t)]^{p_{k}}, ~~\forall~ t \geq 0,   $$
where $\Phi_{p_{k}}^{-1}$ is the inverse function of $\Phi_{p_{k}}$. Since  $g\in L^{\Phi_{3}}\left(\mathbb{T}\right)$, according to Jensen's inequality, we have
\begin{align*}
\Phi_{p_{k}}\left(\frac{p_{k}}{2\pi}\int_{0}^{2\pi}\log^{+}\left(|\Phi_{k}^{-1}\circ\Phi_{3}\left(|g(e^{it})|\right)dt \right) \right)
 &\leq \Phi_{p_{k}}\left( \frac{1}{2\pi}\int_{0}^{2\pi}\left|\Phi_{k}^{-1}\circ\Phi_{3}\left(|g(e^{it})|\right)\right|^{p_{k}}dt \right) \\
&= \Phi_{p_{k}}\left( \frac{1}{2\pi}\int_{0}^{2\pi}\Phi_{p_{k}}^{-1}(\Phi_{3}(|g(e^{it})|))dt \right)\\
&\leq  \frac{1}{2\pi}\int_{0}^{2\pi}\Phi_{p_{k}}\left(\Phi_{p_{k}}^{-1}(\Phi_{3}(|g(e^{it})|))\right)dt \leq 1,
\end{align*}
where $\log^{+}(s)=\max\{0, \log(s)\}$. It follows that
 \begin{equation}\label{eq:inegaaqedehay}
\frac{1}{2\pi}\int_{0}^{2\pi}\log^{+}(|g_{k}(e^{it})|)dt< +\infty,
\end{equation}
since $\log^{+}(st) \leq \log^{+}(s) + \log^{+}(t)$. Let us now show that 
$$  \frac{1}{2\pi}\int_{0}^{2\pi}\log^{-}(|g_{k}(e^{it})|)dt< +\infty,
  $$
where $\log^{-}(s)=\max\{0, -\log(s)\}$. Since $|g| = g_{1}g_{2}$, for almost all $t \in \mathbb{R}$, we have 
\begin{align*}
\log^{-}(|g_{1}(e^{it})|)&= \log^{+}\left(\frac{1}{|g_{1}(e^{it})|}\right)= \log^{+}\left(|g_{2}(e^{it})|\times\frac{1}{|g(e^{it})|}\right) \\
&\leq  \log^{+}(|g_{2}(e^{it})|)+ \log^{+}\left(\frac{1}{|g(e^{it})|}\right) \\ &=\log^{+}(|g_{2}(e^{it})|)+\log^{-}(|g(e^{it})|) \\
 &\leq\log^{+}(|g_{2}(e^{it})|)+|\log(|g(e^{it})|)|,  
\end{align*}
since $|\log(s)|=\log^{-}(s) +\log^{+}(s)$. We deduce that
$$  \log^{-}(|g_{1}(e^{it})|) \leq  \log^{+}(|g_{2}(e^{it})|)+|\log(|g(e^{it})|)|.  $$
Likewise, we also show that 
$$  \log^{-}(|g_{2}(e^{it})|) \leq  \log^{+}(|g_{1}(e^{it})|)+|\log(|g(e^{it})|)|.  $$
Since Relation (\ref{eq:inegaaqedehay}) is satisfied and  $\log|g| \in L^{1}\left(\mathbb{T}\right)$, we deduce that
\begin{equation}\label{eq:inegaehay}
\frac{1}{2\pi}\int_{0}^{2\pi}\log^{-}(|g_{k}(e^{it})|)dt< +\infty.
\end{equation}
It follows that
$$  \frac{1}{2\pi}\int_{0}^{2\pi}|\log(|g_{k}(e^{it})|)|dt = \frac{1}{2\pi}\int_{0}^{2\pi}\log^{-}(|g_{k}(e^{it})|)dt+\frac{1}{2\pi}\int_{0}^{2\pi}\log^{+}(|g_{k}(e^{it})|)dt < +\infty,
   $$
thanks to Relations (\ref{eq:inegaaqedehay}) and (\ref{eq:inegaehay}). Therefore, the outer function $O_{g_{k}}$ associated with $g_{k}$ belongs to $H^{\Phi_{k}}(\mathbb{D})$ and $\|O_{g_{k}}\|_{H^{\Phi_{k}}}^{lux}=\|g_{k}\|_{L^{\Phi_{k}}}^{lux}$,  according to Proposition \ref{pro:mainfpmqaaq5}.   For $z \in \mathbb{D}$, put  
$$  G_{1}(z) =O_{g_{1}}(z)   \hspace*{0.5cm}\textrm{and} \hspace*{0.5cm} G_{2}(z) =B(z)S_{\sigma}(z)O_{g_{2}}(z).  $$
By construction, $G_{1}$ and $G_{2}$ are analytic functions on $\mathbb{D}$ such that  $G_{1} \in H^{\Phi_{1}}(\mathbb{D})$ and $G_{2} \in H^{\Phi_{2}}(\mathbb{D})$. Moreover, $\|G_{1}\|_{H^{\Phi_{1}}}^{lux}=\|g_{1}\|_{L^{\Phi_{1}}}^{lux}$ and $\|G_{2}\|_{H^{\Phi_{2}}}^{lux}=\|g_{2}\|_{L^{\Phi_{2}}}^{lux}$, since  $B S_{\sigma}$  is inner function on $\mathbb{D}$. It follows that, for all $z\in \mathbb{D}$, 
$$ G(z) =B(z)S_{\sigma}(z)O_{g_{1}}(z)O_{g_{2}}(z)=   G_{1}(z) G_{2}(z)$$ 
and
$$  \|G_{1}\|_{H^{\Phi_{1}}}^{lux}\|G_{2}\|_{H^{\Phi_{2}}}^{lux}=\|g_{1}\|_{L^{\Phi_{1}}}^{lux}\|g_{2}\|_{L^{\Phi_{2}}}^{lux} \leq \|g\|_{L^{\Phi_{3}}}^{lux} = \|G\|_{H^{\Phi_{3}}}^{lux} \lesssim \|G_{1}\|_{H^{\Phi_{1}}}^{lux}\|G_{2}\|_{H^{\Phi_{2}}}^{lux},
  $$
thanks to Relation (\ref{eq:inegaaqlitedehay}) and Lemma \ref{pro:main 5aqp1aqqq2pl}.
\epf

\subsection{Applications to Hankel Operators.} 

\subsubsection{Proofs of  Theorem \ref{pro:mainfqmaqppaaqqmaqq5} and Theorem \ref{pro:madgaaqqmaqq5}.} 

\begin{lemma}\label{pro:mainfqmaqpq5}
Let  $\Phi_{1}$ be a growth function of lower type $p_{1}$ and  $\Phi_{2} \in \Delta_{2} \cap \nabla_{2}$  a convex  growth function. Let $\Phi_{3}$ a  positive function   on $[0, \infty)$ such that
\begin{equation}\label{eq:deltAQa2}
\Phi_{3}^{-1}(t) =\Phi_{1}^{-1}(t)\Psi_{2}^{-1}(t),~ ~\forall~ t >  0,\end{equation}
where   $\Psi_{2}$  the complementary function of $\Phi_{2}$. The following assertions are satisfied:
\begin{itemize}
\item[(i)] If  $ b_{\Phi_{2}} < p_{1}$ then
 $\Phi_{3}$  is  a convex  growth function belongs to $\Delta_{2} \cap \nabla_{2}$.
\item[(ii)] If  $\Phi_{1}$ is also of upper type $q_{1}$ and  $0<p_{1} \leq q_{1} \leq  a_{\Phi_{2}}$ then   $\Phi_{3}$ is  a growth function of both lower type $p_{3}$ and upper type $q_{3}$ such that  $0<p_{3} \leq q_{3} \leq 1$.
\end{itemize} 
\end{lemma}

\begin{proof} 
Since  $\Phi_{2} \in \Delta_{2} \cap \nabla_{2}$, we deduce that  $\Psi_{2} \in \Delta_{2} \cap \nabla_{2}$. Moreover,  $ \frac{1}{a_{\Psi_{2}}}+\frac{1}{b_{\Phi_{2}}} =1$ and
$\Phi_{2}^{-1}(t)\Psi_{2}^{-1}(t) \sim t$, for all  $t>0$.

\medskip

$i)$  According to Lemma \ref{pro:mainfqmpq5}, the function  $\Phi_{3}$ is a growth function of lower type $p_{3}:= \left( \frac{1}{p_{1}} + \frac{1}{a_{\Psi_{2}}}\right)^{-1}$ and upper type $q_{3}:=b_{\Psi_{2}}$. Since  $b_{\Phi_{2}}<p_{1}$, we deduce that
$$ \frac{1}{p_{3}} =\frac{1}{p_{1}} + \frac{1}{a_{\Psi_{2}}} <  \frac{1}{b_{\Phi_{2}}} + \frac{1}{a_{\Psi_{2}}}= 1.   $$
It follows that $\Phi_{3}$ is a growth function of both lower type $p_{3}$ and upper type  $q_{3}$ such that  $1< p_{3} \leq q_{3}< \infty$. It follows that,  $\Phi_{3}$ belongs to $\Delta_{2} \cap \nabla_{2}$, thanks to Proposition \ref{pro:maiaqaq1qaqks8}.

\medskip

$ii)$ The function  $\Phi_{3}$ is a growth function of lower type $p_{3}:= \left( \frac{1}{p_{1}} + \frac{1}{a_{\Psi_{2}}}\right)^{-1}$ and upper type  $q_{3}:= \left( \frac{1}{q_{1}} + \frac{1}{b_{\Psi_{2}}}\right)^{-1}$, thanks to Lemma \ref{pro:mainfqmpq5}.  Since $q_{1} \leq a_{\Phi_{2}}$, we have
$$\frac{1}{q_{3}} =\frac{1}{q_{1}} + \frac{1}{b_{\Psi_{2}}} \geq \frac{1}{a_{\Phi_{2}}} + \frac{1}{b_{\Psi_{2}}}  =1.  $$
We deduce that   $0<p_{3}\leq q_{3} \leq 1$. 
\end{proof}

From the duality result in \cite{raoren}, we obtain the following result.

\begin{theorem}\label{pro:main2aaopapmp}
Let  $\Phi \in \Delta_{2} \cap \nabla_{2}$ be a convex growth function. The topological dual of $H^{\Phi}(\mathbb{D})$,   $\left(H^{\Phi}(\mathbb{D})\right)^{*}$    is isomorphic to $H^{\Psi}(\mathbb{D})$, in the sense that, for all  $T\in \left(H^{\Phi}(\mathbb{D})\right)^{*}$, there is a unique $G\in H^{\Psi}(\mathbb{D})$ such that
 $$   T(F)=\langle F,G  \rangle := \lim_{r \to 1} \frac{1}{2\pi}\int_{-\pi}^{\pi}F(re^{i\theta})\overline{G(re^{i\theta})}d\theta, ~~ \forall~F\in H^{\Phi}(\mathbb{D}). 
   $$
Moreover, 
$$ \|G\|_{H^{\Psi}}^{lux} \approx \sup\{ |\langle F,G  \rangle|:    F\in H^{\Phi}(\mathbb{D})\hspace*{0.15cm} \text{with}\hspace*{0.15cm} \|F\|_{H^{\Phi}}^{lux} \leq 1   \}.   $$
\end{theorem}

As pointed out in \cite{BoGre, BoSehb}, from Viviani's results \cite{Viviani}, $BMOA(\varrho)$ spaces appear as duals of particular Hardy-Orlicz spaces.

\begin{theorem}\label{pro:main2aaop}
Let $\Phi$ be a growth function of lower type  $0< p \leq 1$ and upper type $1$ respectively. The topological dual of $H^{\Phi}(\mathbb{D})$,   $\left(H^{\Phi}(\mathbb{D})\right)^{*}$    is isomorphic to $BMOA(\varrho)$, where 
$\varrho(t):=\frac{1}{t\Phi^{-1}(1/t)}$,  in the sense that, for all  $T\in \left(H^{\Phi}(\mathbb{D})\right)^{*}$, there is a unique $G\in BMOA(\varrho)$ such that
 $$   \langle F,G  \rangle := \lim_{r \to 1} \frac{1}{2\pi}\int_{-\pi}^{\pi}F(re^{i\theta})\overline{G(re^{i\theta})}d\theta, ~~ \forall~F\in H^{\Phi}(\mathbb{D}). 
   $$
Moreover, 
$$ \|G\|_{BMOA(\varrho)} \approx \sup\{ |\langle F,G  \rangle|:    F\in H^{\Phi}(\mathbb{D})\hspace*{0.15cm} \text{with}\hspace*{0.15cm} \|F\|_{H^{\Phi}}^{lux} \leq 1   \}.   $$
\end{theorem}

\proof[Proof of Theorem \ref{pro:mainfqmaqppaaqqmaqq5}.]
Since $\Phi_{2}$ is convex (as composed of two convex functions), we deduce that $\Phi_{2}$ belongs to $\Delta_{2} \cap \nabla_{2}$ and and $p_{2} \leq a_{\Phi_{2}} \leq b_{\Phi_{2}}\leq q_{2}$, thanks to Proposition \ref{pro:maiaqaq1qaqks8}

\medskip

Let $\Phi_{4}$ a  positive function   on $[0, \infty)$ such that
$$  \Phi_{4}^{-1}(t) =\Phi_{1}^{-1}(t)\Psi_{2}^{-1}(t),~ ~\forall~ t >  0,  $$
where   $\Psi_{2}$  the complementary function of $\Phi_{2}$. Since  $ b_{\Phi_{2}} < p_{1}$, we deduce that  $\Phi_{4}$ is a convex growth function belongs to $\Delta_{2} \cap \nabla_{2}$, according  to Lemma \ref{pro:mainfqmaqpq5}.  For $ t> 0$,  we have
 $$  \Phi_{4}^{-1}(t).\Phi_{3}^{-1}(t)=\Phi_{1}^{-1}(t).\Psi_{2}^{-1}(t) \times\frac{\Phi_{2}^{-1}(t)}{\Phi_{1}^{-1}(t)} = \Psi_{2}^{-1}(t).\Phi_{2}^{-1}(t) \sim t.
   $$
It follows that $\Phi_{3}$ is   the complementary function of $\Phi_{4}$. Therefore,  $H^{\Phi_{4}}(\mathbb{D})$ is the topological dual of $H^{\Phi_{3}}(\mathbb{D})$, thanks to Theorem \ref{pro:main2aaopapmp}.

\medskip

Suppose that $b$ belongs to $H^{\Phi_{3}}(\mathbb{D})$ and show that $h_{b}$ is bounded from $H^{\Phi_{1}}(\mathbb{D})$ to $H^{\Phi_{2}}(\mathbb{D})$.

Let $F\in H^{\Phi_{1}}(\mathbb{D})$ and  $G\in H^{\Psi_{2}}(\mathbb{D})$ such that $\|G\|_{H^{\Psi_{2}}}^{lux}\leq 1$.  We have
\begin{align*}
 |\langle h_{b}(F), G   \rangle|&= |\langle \mathcal{P}(b\overline{F}), \mathcal{P}(G)   \rangle| =  |\langle b, FG   \rangle| \\
 &\lesssim \|b\|_{H^{\Phi_{3}}}^{lux}\|FG\|_{H^{\Phi_{4}}}^{lux}\\
 &\lesssim \|b\|_{H^{\Phi_{3}}}^{lux}\|F\|_{H^{\Phi_{1}}}^{lux}\|G\|_{H^{\Psi_{2}}}^{lux} \\
 &\lesssim \|b\|_{H^{\Phi_{3}}}^{lux}\|F\|_{H^{\Phi_{1}}}^{lux}.
   \end{align*}
We deduce that  $h_{b}(F) \in H^{\Phi_{2}}(\mathbb{D})$ and 
\begin{equation}\label{eq:ineitedehay}
\|h_{b}(F)\|_{H^{\Phi_{2}}}^{lux}\lesssim \|b\|_{H^{\Phi_{3}}}^{lux}\|F\|_{H^{\Phi_{1}}}^{lux}.
\end{equation}

\medskip

Conversely, suppose that the Hankel operator $h_{b}: H^{\Phi_{1}}(\mathbb{D})\longrightarrow H^{\Phi_{2}}(\mathbb{D})$ is bounded  and prove that $b$ belongs to  $H^{\Phi_{3}}(\mathbb{D})$. 

Let  $F\in H^{\Phi_{4}}(\mathbb{D})$ such that $\|F\|_{H^{\Phi_{4}}}^{lux}\leq 1$. According to Theorem \ref{pro:main2aqop}, there exist  $F_{1}\in H^{\Phi_{1}}(\mathbb{D})$ and $F_{2}\in H^{\Psi_{2}}(\mathbb{D})$ such that
   $F=F_{1}F_{2}$ and  $\|F\|_{H^{\Phi_{4}}}^{lux}\approx \|F_{1}\|_{H^{\Phi_{1}}}^{lux}\|F_{2}\|_{H^{\Psi_{2}}}^{lux}$. We have 
\begin{align*}
  |\langle b, F   \rangle| &= |\langle b, F_{1} F_{2}  \rangle| 
  = |\langle h_{b}(F_{1}),  F_{2}  \rangle| \\
  &\lesssim \|h_{b}(F_{1})\|_{H^{\Phi_{2}}}^{lux} \|F_{2}\|_{H^{\Psi_{2}}}^{lux} \lesssim \|h_{b}\| \|F_{1}\|_{H^{\Phi_{1}}}^{lux} \|F_{2}\|_{H^{\Psi_{2}}}^{lux} \\
  &\lesssim \|h_{b}\|\|F\|_{H^{\Phi_{4}}}^{lux}\lesssim \|h_{b}\|.
   \end{align*}
We deduce that   $b \in H^{\Phi_{3}}(\mathbb{D})$ and 
 \begin{equation}\label{eq:ineitaqedehay}
  \|b\|_{H^{\Phi_{3}}}^{lux}\lesssim \|h_{b}\|.
 \end{equation}
It ends the proof.   
\epf

\proof[Proof of Theorem \ref{pro:madgaaqqmaqq5}.]
Let $\Phi_{3}$ a  positive function   on $[0, \infty)$ such that
$$  \Phi_{3}^{-1}(t) =\Phi_{1}^{-1}(t)\Psi_{2}^{-1}(t),~ ~\forall~ t >  0,  $$
where   $\Psi_{2}$  the complementary function of $\Phi_{2}$. Since $0<p_{1} \leq q_{1} \leq  a_{\Phi_{2}}$, we deduce that   $\Phi_{3}$ is  a growth function of both lower type $p_{3}$ and upper type $q_{3}$ such that  $0<p_{3} \leq q_{3} \leq 1$, thanks to Lemma \ref{pro:mainfqmaqpq5}. For $ t> 0$, we have
$$\varrho(t):=\frac{\varrho_{1}(t)}{\varrho_{2}(t)}= \frac{\Phi_{2}^{-1}(1/t)}{\Phi_{1}^{-1}(1/t)} = \frac{\Phi_{2}^{-1}(1/t)\Psi_{2}^{-1}(1/t)}{\Phi_{1}^{-1}(1/t)\Psi_{2}^{-1}(1/t)} \sim \frac{1/t}{\Phi_{1}^{-1}(1/t)\Psi_{2}^{-1}(1/t)}= \frac{1}{t\Phi_{3}^{-1}(1/t)}.  $$
It follows that,   $H^{\Psi_{2}}(\mathbb{D})$ and  $BMOA(\varrho)$ are the respective dual spaces of spaces  $H^{\Phi_{2}}(\mathbb{D})$ and  $H^{\Phi_{3}}(\mathbb{D})$, according to Theorem \ref{pro:main2aaopapmp} and Theorem \ref{pro:main2aaop}. 

\medskip

The rest of the Proof is identical to that of the proof of Theorem \ref{pro:mainfqmaqppaaqqmaqq5}. Consequently, it will be omitted.
\epf

\bibliographystyle{plain}
 
\end{document}